\documentclass[11pt,english]{amsart}
\usepackage[T1]{fontenc}
\usepackage[utf8]{luainputenc}
\usepackage{xcolor}
\usepackage{pdfcolmk}
\usepackage{amsthm}
\usepackage{amsmath}
\usepackage{stmaryrd}
\usepackage{stackrel}
\usepackage{setspace}
\PassOptionsToPackage{normalem}{ulem}
\usepackage{ulem}
\usepackage{tikz}
\makeatletter

\providecolor{lyxadded}{rgb}{0,0,1}
\providecolor{lyxdeleted}{rgb}{1,0,0}

\numberwithin{figure}{section}
\numberwithin{equation}{section}
\numberwithin{table}{section}

  \input amssymb.sty
\usepackage{etoolbox}
\patchcmd{\thebibliography}{\section*}{\section}{}{}

\usepackage{algpseudocode} \usepackage{algorithm}
\usepackage{algpseudocode,algorithm}
\usepackage{lipsum}
\usepackage{caption}
\usepackage{graphics}

\usepackage{amsmath}
\usepackage{amssymb}
\usepackage[all]{xy}

\usepackage{epsfig}\usepackage{youngtab}

\newcommand{\ef}{\end{equation}}
\chardef\bslash=`\\ 

\newdir{:=}{{}}
\newcommand*\colvec[3][]{
    \begin{pmatrix}\ifx\relax#1\relax\else#1\\\fi#2\\#3\end{pmatrix}
}

\newtheorem*{thm*}{Theorem}

\newtheorem{lem}{Lemma}[section]
\newtheorem*{lem*}{Lemma}
\newtheorem{corl}{Corollary}[lem]

\newtheorem*{corl*}{Corollary}
\newtheorem{prop}{Proposition}[section]

\newtheorem{prop*}{Proposition}
\newtheorem{Goal*}{Goal}

\theoremstyle{definition}
\newtheorem{defn}{Definition}[section]
\newtheorem{examp}{Example}
\newtheorem*{examp*}{Example}
\newtheorem*{remark*}{Remark}
\newtheorem*{CC*}{Crossover Conjecture}
\newtheorem*{Note*}{Note}
\newtheorem*{defn*}{Definition}

 \theoremstyle{remark}

 \renewcommand{\sectionmark}[1]{}

\newcommand{\Inv}{\operatorname{Inv}}
\newcommand{\la}{\langle}
\newcommand{\ra}{\rangle}

\newcommand{\thrh}{\frac{3}{2}}

\newcommand{\sevh}{\frac{7}{2}}
\newcommand{\eleh}{\frac{11}{2}}

\newcommand{\defect}{\operatorname{def}}

\newcommand{\diag}{\operatorname{diag}}

\newcommand{\lm}{\lambda}
\newcommand{\Lm}{\Lambda}

\newcommand{\hub}{\operatorname{hub}}
\newcommand{\cont}{\operatorname{cont}}

\renewcommand{\a}{\alpha}

\makeatother

\usepackage{babel}

\begin{document}
\title[Spin Multipartitions]{Spin Multipartitions}

\author{Ola Amara-Omari}
\address{ Bar-Ilan University, Ramat-Gan, Israel}
\email{Omariao@biu.ac.il}
\author{Mary Schaps}
\address{ Bar-Ilan University, Ramat-Gan, Israel,}
\email{mschaps@macs.biu.ac.il}

\subjclass[2020] {17B10, 17B65, 05E10, 20C08}
\maketitle

 \begin{abstract} 
We give an algorithm to construct spin multipartitions. The Fock space for a general dominant integral weight is the tensor product of Fock spaces $\mathcal{F}_{i}$. Leclerc and Thibon in \cite{LT} showed that  $\mathcal{F}_0$ is a module over the quantized enveloping algebra.  We show that the relations required of a module for the quantized enveloping algebra do  hold for the remaining $\mathcal{F}_i$ with our combinatorics.
\end{abstract}

\maketitle
\section{INTRODUCTION}

To explain what we mean by a ``spin multipartition'', we give a very brief review of the theory of multipartitions for the case of type A. 
The association of partitions with simple representations of the symmetric groups can be traced back to Schur \cite{Sch}. The behavior of modules under restriction was determined by Littlewood-Richardson branching rules.  Dipper and James \cite{DJ} discovered a close connection between the representation theory of the symmetric group over a field of characteristic $p$ and the representations of an Iwahori-Hecke algebra over $\mathbb C$ at a $p$-th root of unity. Jimbo et al. \cite{JMMO} introduced higher level Fock spaces.
Kleshchev associated residues to the nodes in the partitions and developed modular branching rules.  Lascoux, Leclerc, and Thibon \cite{LLT} noticed that these modular branching rules provided a connection to the Kashiwara crystal graph as given in \cite{MM} and conjectured a close connection between the canonical bases of modules over affine Kac-Moody algebras and projective indecomposable modules over the Hecke algebra at a $p$-th root of unity. Ariki \cite{A1} proved the LLT conjectures, not only for the symmetric groups, but for the Hecke algebras of the complex reflection groups of type $G(\ell,1,d)$. The irreducible modules are no longer labeled by partitions, but by multipartitions \cite{AM}. 

We now turn to the spin case.  The projective representations of the symmetric group are representations of a central extension of the symmetric group by a cyclic group of order two. Those representations for which the generator $z$ of the kernel takes the value $1$ are called ordinary representations and the spin representations are those for which $z$ takes the value $-1$.  The theory of spin representations lags behind that of the symmetric groups, but what can be done for the ordinary representations can usually be extended eventually to spin representations. A theory of strict partitions labeling spin representations goes back to Morris and Yasseen, \cite{MY1, MY2}. The partitions representing the irreducible modules over a field of characteristic $p$ are not, as in the ordinary case, a subset of the partitions in the characteristic case.  Instead, one allows partitions that are not strict, provided that the repeated rows are divisible by $p$.
Leclerc and Thibon \cite{LT} conjectured that the spin representations also correspond to the simple modules of a highest weight representation $V(\Lambda)$ for the case $\Lambda = \Lambda_0$ of an  affine Lie algebra, and this was eventually proven by Brundan and Kleshchev in \cite{BK1}.

One can ask what happens to the labeling of the crystal base for more general highest weights $\Lambda$? In analogy to the case of cyclotomic Hecke algebras in type A, one would expect some form of multipartition, but a naive attempt by the authors ten years ago to define spin multipartitions using only strict partitions failed dismally because it could not reproduce the crystal structure. 
  
Following conventions for spin blocks of the symmetric groups, we fix an odd number $h=1+2n$ with $n \geq 1$.  
We fix a dominant integral weight $\Lambda=a_0\Lambda_0+a_1\Lambda_1+\dots \a_n\Lambda_n$ for the twisted affine Lie algebra $ A_{2n}^{(2)}$ as defined in \cite{Ka}.  For the given  highest weight $\Lambda$  the  weights in $V(\Lm)$ are of the form $\lambda=\Lambda-\alpha$ where $\alpha$ is a sum with non-negative coefficients of simple roots.

We now go back to Ariki's original introduction of the multipartitions
using tensor products \cite{A1,A2}. Since type A has a Dynkin diagram with cyclic symmetry, all the highest weight representations $V(\Lambda_i)$ are isomorphic.  This is not at all the case for the spin representations.
While it will still be true that we will need $a_0$ $h$-strict, $h$-restricted partitions, for the remaining residues we do not at all want strict partitions.  
The partitions we need will be non-strict, as in type A, but the residues will follow the order used in the case of $\Lm=\Lm_o$ in \cite{LT}.

The main result of this paper is the following:
\begin{thm*}
	Let $\mathfrak{g}$ be the affine Lie algebra of type $A^{(2)}_{2n}$ for a positive integer $n$. The level one Fock spaces $\mathcal{F}_i$ given by $i$-corner partitions are modules for the quantum enveloping algebra
	$U_q(\mathfrak{g})$.
\end{thm*}	

We hope, as in Ariki and Mathas\cite{AM}m to string the level one partitions together into genuine multipartitions using the Hopf algebra structure of the quantum enveloping algebra.
   
\section{DEFINITIONS AND NOTATION}

Let $\mathfrak g$ be the affine Lie algebra $A^{(2)}_{2n}$ as in \cite [Ch.4]{Ka},  and let $C=[a_{ij}]$ be the Cartan matrix. For $n=1$, the Cartan matrix is
\[
C=
\begin{bmatrix}
2&-4\\
-1&2
\end{bmatrix}.
\]
For the general case $n\geq 2$, we have
\[C=
\begin{bmatrix}
2&-2&0&\dots&0&0\\
-1&2&-1&0&\dots&0\\
0&-1&2&\ddots&\ddots&\vdots\\
\vdots&\ddots&\ddots&\ddots&-1&0\\
0&\dots&0&-1&2&-2\\
0&0&\dots&0&-1&2
\end{bmatrix}.
\]

 Let $Q$ be the $\mathbb Z$-lattice generated by the simple roots
$\a_0,\dots,\a_{n}$. Kac \cite{Ka} uses the notation $\a_i^\vee$ for the coroots, but we will denote the corresponding coroots by $h_0, h_1, \dots h_\ell$. 
 The pairing between a root and a coroot is given by the appropriate entry in the Cartan matrix. 
\[
\langle h_i, \a_j \rangle =a_{ij}, i,j=0,1,\dots,n.
\]

The Cartan matrix is not symmetric, but it is symmetrizable, the symmetrizing matrix being the diagonal matrix $D=\diag(\frac{1}{2}, 1,\dots, 1, 2)$. Thus in the symmetric product, $(\a_0\mid \a_0)=1$,
	$(\a_i \mid \a_i)=2$ for $i=1,2,\dots,n-1$ and $(\a_n \mid \a_n)=4$.
Other values of this symmetric product of interest to us will be $(\a_0\mid\a_1)=-1$, $(\a_{n-1}\mid\a_n)=-2$. Also $(\Lambda_i \mid \a_j)=\delta_{ij}$, for all
$i,j=0,1,\dots,\ell$.
In this paper, we will be working with a fixed non-zero
 dominant integral weight, that is to say, a weight $\Lambda$ such that $\langle h_i, \Lambda \rangle $ is a nonnegative integer for every $i$.  

The Cartan subalgebra of the Lie algebra $\mathfrak g$ is generated by 
$h_0, h_1, \dots, h_\ell$ and an element $d$ called the \textit{scaling element},
which satisfies 
\[
\langle d, \alpha_i \rangle =0, 1 \leq i \leq \ell, \langle d,\alpha_0 \rangle =1.
\]

The center of  $\mathfrak g$ is one-dimensional, spanned by 
\[
c = h_0+2\sum_{i=1}^n h_i,
\] 
which is called the \textit{canonical central element}. The \textit{level} of a weight $\lambda$ is $\langle \lambda, c \rangle$. 

Let  $Q_+$ be the subset of $Q$ in which all coefficients of the $\a_i$ are nonnegative.
The \textit{null root} $\delta$ is $2\a_0+2\a_1\dots+2\a_{n-1}+\a_{n}$.
This formula is chosen for the null root because the vector $(2,2,2,\dots,2,1)$, as a column vector, produces $0$ when multiplied by the Cartan matrix on the left. We have  $(\a_i\mid \delta)=0$. 

 We define the fundamental weights $\Lambda_j, 0 \leq j \leq \ell$ together with the null root  to be weights dual to the coroots and the scaling element, so chosen that $ \langle h_i, \delta \rangle = 0$ for all $i$, $\langle d, \delta \rangle=1$, $ \langle d, \Lambda_j \rangle = 0$ for all $j$, and
$ \langle h_i, \Lambda_j \rangle = \delta_{ij}$, where $\delta_{ij}$ is the Kronecker delta, non-zero only when $i=j$ and then equal to $1$.

The $\mathbb Z$-lattice $P$ of weights of the affine Lie algebra spans a real vector space that has two different bases.  One is given by the fundamental weights together with the null root, $\Lambda_0,\dots, \Lambda_{\ell}, \delta$, 
and one is given by $\Lambda_0, \a_0,\dots,\a_{\ell}$.  We will usually use the first basis for our weights, whereas Kac \cite{Ka} prefers the second.
 In addition, $(\Lambda_0\,\mid \Lambda_0 )=(\delta \mid \delta) =0$ and  $\la \Lm,c \ra=r$ the level defined in the previous paragraph.

 Our modules will all be left modules. Let $\Lambda$ be a fixed non-zero dominant integral weight. Let $V(\Lambda)$ be a highest weight module with  
 highest weight $\Lambda$, and let $P(\Lambda)$ be the set of weights of $V(\Lambda)$. In affine type A, if $\Lambda=a_0\Lambda_0+a_1\Lambda_1+\dots+ a_n\Lambda_n$ for nonnegative integers $a_i$, the level is given by the formula $r=a_0+2a_1+ \dots +2a_n$ . We have $r>0$ because $\Lambda$ is non-zero.

  As in \cite [\S3.3]{Kl}, we define the defect of a weight $\lambda=\Lambda -\alpha$ by 

\[
\defect(\lambda)=\frac{1}{2}((\Lambda \mid \Lambda)-  (\lambda \mid \lambda))=(\Lambda \mid \alpha)-\frac{1}{2}(\alpha \mid \alpha).
\]

Since we are in a highest weight module, we always have $(\Lambda \mid \Lambda) \geq (\lambda \mid \lambda)$ [\cite{Ka}, Prop. 11.4(a)].  The definition of the defect is general for all affine types, and our affine Lie algebra is the only one for which the defect may be a half-integer. There is an alternative definition of the defect which multiplies it by $2$ to get an integer, but we prefer to use the general definition. The weights of defect $0$ are those lying in the Weyl group orbit of $\Lambda$.
Every weight $\lambda \in P(\Lambda)$ has the form $\Lambda-\alpha$, for $\alpha \in Q_+$. We will need
\begin{align*}
\defect(\lambda-t \delta)=&(\Lambda \mid \alpha+t\delta)-\frac{1}{2}(\alpha +t\delta \mid \alpha +t\delta)\\
=&\defect(\lambda)+tr.
\end{align*}

\begin{defn}\label{deg} Let $\lambda = \Lambda - \alpha$ with $\alpha \in Q_+$. If $\alpha=\sum_{i=0}^\ell c_i \alpha_i$, where all $c_i$ are nonnegative, then the vector 
\[
\cont(\lambda)=(c_0.\dots,c_\ell)
\]
is called the \textit{content} of $\lambda$. The \textit{degree} $\deg(\lambda)$ of the weight $\lambda$ is the integer $n=\sum_{i=0}^\ell c_i$.
\end{defn}

Define
\[
\max (\Lambda)=\{\lambda \in P(\Lambda) \mid \lambda + \delta \not\in P(\Lambda)\},
\] 
and by \cite[\S 12.6.1]{Ka}, every element of 
$P(\Lambda)$ is of the form $\{y-t\delta \mid y \in \max(\Lambda), t \in \mathbb Z_{ \geq 0}\}$.

\begin{defn}\label{M} Let $W$ denote the Weyl group
	\[
	W=T \rtimes  \mathring{W}
	\] 
	where $\mathring{W}$ is the finite Weyl group of type $B_{2n}$ and   the elements of the abelian normal subgroup $T$ are transformations of the form 
	
	\[
	t_\alpha(\zeta)=\zeta+r\alpha -((\zeta|\alpha)+\frac{1}{2}(\alpha|\alpha)r)\delta,
	\]
	\noindent where the weights  $\a$  are taken from a from a lattice $M$ generated by the long roots with half-integer coefficients \cite[\S 6.5.2]{Ka}.
\end{defn}

The role of $M$ is explained before \cite[Prop. 2.11]{FSv2}, in the version of \cite{BFS} which treats the spin case in greatest detail.

\begin{defn}\label{hub} For a dominant integral weight $\Lambda$ and for any weight $\lambda$ in the set of weights  $P(\Lambda)$ of the highest weight module $V(\Lm$), we let $\hub(\lambda)=[\theta_0,\dots, \theta_{\ell}]$ be the \textit{hub} of $\lambda$, where 
	\[
	\theta_i=\la h_i, \lambda \ra.
	\]
	\noindent Let the real weight space be the $\mathbb{R}$-vector space whose basis is the set of weights of $\mathfrak{g}$. Let $U$ be the vector subspace of the real weight space generated by the fundamental weights. The hub is the projection (with respect to the decomposition $U \oplus \mathbb{R}\delta$) of the weight of $\lambda$ onto $U$. 
\end{defn} 
We write the hub with square brackets to distinguish the hub from the content. The original definition given by Fayers in \cite{Fa} was the negative of this one.

\section{DEFECTS IN THE REDUCED CRYSTAL GRAPH}

\begin{defn}\label{block-reduced}
	The set $P(\Lambda)$ can be taken as the set of vertices of a graph $\widehat P(\Lambda)$ which we will call the \textit{block-reduced 
		crystal}. Two weights $\mu,\nu \in P(\Lambda)$ will be connected by an edge of residue $i$ for $i=0,1,\dots, n$   if $\mu -\nu = \pm \alpha_i$ for a simple root $\alpha_i$. A maximal set of vertices connected by edges of residue $i$ will be called an \textit{$i$-string}
	 and all $i$-strings are finite because $V(\Lambda)$ is integrable.
\end{defn}
The block-reduced crystal was used in \cite[Def. 2.6]{AS},   or \cite[\S 2]{BFS}. 
 We now make use of the hub and its components $\theta_i$ from Def \ref{hub}.

\begin{lem}\label{defect}For a dominant integral weight 
	$\Lambda=a_0\Lambda_0+a_1\Lambda_1+\dots+a_n \Lambda_{n}$, if $\eta=\Lambda-\alpha$ for $\alpha \in Q_+$ has a nonnegative $i$-component $u=\theta_i$ in $\hub(\eta)$, then $\eta, \eta-\alpha_i,\dots,\eta-u\alpha_i$ are vertices in an $i$-string with $s_i(\eta)=\eta-u\alpha_i$. We recall that the symmetrizing matrix $D=\diag(d_0,d_1,\dots,d_n)$ where $d_0=\frac{1}{2}, d_1=1,\dots, d_{n-1}=1,d_n=2$,
	\begin{enumerate}
		\item   The defect of $\lambda=\eta-k\alpha_i$ for $0 \leq k \leq u$ is
		$	\defect(\eta)+d_ik(u-k)$.
		\item  The $i$-components of the hub descend by $2$ as we move down the $i$-string,  so the absolute values of the hub components
		 $\theta_i=\langle h_i, \lambda  \rangle$ decrease as the defect increases and then increase as the defect decreases.
	\end{enumerate}	 
			  In summary,
		\begin{center}
			\begin{tabular}{ | l || c | c | c | c | c | c | }
				\hline
				$\lambda$ &$\eta$ & $\eta-\a_i$&$\dots$ & $\eta-k\a_i$ &$\dots$ & $\eta-u\a_i$ \\ \hline \hline
				$\defect(\lambda)$ &$\defect(\eta)$&$\defect(\eta)+d_i(u-1)$&$\dots$&$\defect(\eta)+d_ik(u-k)$&$\dots$ & $\defect(\eta)$\\ \hline
				$\langle h_i,\lambda \rangle$ &u&u-2&\dots&u-2k&\dots&-u \\ \hline					
				\hline
			\end{tabular}
		\end{center}
	
\end{lem}

\begin{proof}		`	
	\begin{enumerate}
		\item 	We compute the defect explicitly for $k\alpha_i$, using $(\alpha_i|\alpha_i)=2d_i$, 
		\begin{align*}
			\defect(\lambda)&=(\Lambda|\alpha+k\alpha_i)
		-\frac{1}{2}(\alpha+k\alpha_i|\alpha+k\alpha_i)\\
		&=(\Lambda|\alpha)
		-\frac{1}{2}(\alpha|\alpha)+(\Lambda|k\alpha_i)
		-(\alpha|k\alpha_i)
		-\frac{1}{2}(k\alpha_i|k\alpha_i)\\
		&=\defect(\eta)+k((\Lambda|\alpha_i)-(\alpha|\alpha_i))-d_ik^2\\
		&=\defect(\eta)+k((\Lambda-\alpha)|\alpha_i)-d_ik^2\\
		&=\defect(\eta)+d_i k\la h_i, \Lambda-\alpha\ra-d_ik^2\\
		&=\defect(\eta)+d_ik(u-k).	
		\end{align*}		
		\item $\langle h_i, -\a_i \rangle = -2$. For $u$ even, the defect is maximal for $k= \frac{u}{2}$, and when $w$ is odd, the maximal values of the defect are obtained at $k=\frac{u-1}{2},\frac{u+1}{2}$.		
	\end{enumerate}
\end{proof}
Recall that $\Lambda=a_0\Lambda_0+a_1\Lambda_1+\dots+a_n \Lambda_{n}$. In a reduced crystal for $\Lambda$, the $i$-string from $\Lambda$ has length $a_i$. The weights $\Lambda-k\alpha_i$ necessarily lie in $\max(\Lambda)$ since  
$\Lambda-k\alpha_i+\delta$ cannot lie in $P(\Lambda)$ because   $k \alpha_i-\delta \notin Q_+$ since $\alpha_j$  is negative for every $j \neq i$. We now give a corollary to Lemma \ref{defect} regarding the crystal in Def. \ref{block-reduced}.
\begin{corl}\label{def}
	In a  block-reduced crystal $\widehat P(\Lambda)$ with $\Lambda=a_0\Lambda_0 +\dots+ a_\ell\Lambda_\ell$, the defect of $\lambda=\Lambda-k\alpha_i$ for $ 0 \leq k \leq a_i$ is $d_ik(a_i-k)$.
\end{corl}
\begin{proof}
	$\defect(\Lambda)=0$ and $\hub(\Lambda)=[a_0,\dots,a_\ell]$, so $w=a_i$.
\end{proof}

\section{THE SPIN MULTIPARTITION ALGORITHM}

As we explained in the introduction, the spin multipartitions require a more complicated algorithm than the type A $e$-restricted multipartitions because we are taking them from a tensor product in which the factors are mutually non-isomorphic.  We could have found ourselves in a situation where each residue requires a different criterion.  Fortunately, this does not appear to be the case.

Let us first recall the residue function we are using, following \cite{LT}. We have an odd integer $h=1+2n$. For any integer $m$ we define $\widehat{m}$ to be either $m-1$ or $-m$ modulo $h$, whichever is smaller when the residues in $\mathbb{Z}/h\mathbb{Z}$ are embedded in $\mathbb{Z}$ as $0,1,2,\dots,h$. 
In the $\Lambda=\Lambda_0$ case, the standard spin combinatorics puts these values in the nodes of each row of the Young diagram. Let $(s,t)$ be the row and column of a node, starting from $(1,1)$. Then the residue placed in node $(s,t)$ is $\hat{t}$.  If, for example, $h=7, n=3$, we would have the residues of each row given by $0,1,2,3,2,1,0,0,1,2,3,2,1,0,0,1,\dots$.

Setting aside for the moment residue $0$, we try a similar procedure taking the spin residue of $k_\ell+s-r+1$. We have to add $1$ because $\hat{1}=0$ and  $\hat{2}=1$. we need 
the spin residue of the node $(1,1,\ell)$ to be $k_\ell$
and the spin residue of $(1,2,\ell)$ to be $k_\ell+1$.  For example,  for a residue $k_i=1$ when $h=7$, we could get a Young diagram with residues

\young(123210012321001,01232100123210,001232,100123,2100)

\noindent with equal residues running down the diagonals as in the type A multipartitions, rather than down the columns as in the standard strict partitions used for the spin representations of the symmetric group.

\begin{defn}
	A partition  is $h$-\textit{restricted} if it does not have 
	$h$ identical columns and is $h$-strict if each partition has no identical rows unless the length of the row is divisible by $h$.
\end{defn}

\begin{defn}
	Let $\mathfrak{g}=A^{(2)}_{2n}$ and let $\Lm-(a_0\Lm_0+\dots+a_n\Lm_n)$ be a positive integral weight for $\mathfrak{g}$. A \textit{spin multipartition} for $\Lm$ is a set of $a_0+a_1+\dots+a_n$ partitions, of which
	\begin{itemize}
		\item  The bottom $a_0$ are $h$-restricted $h$-strict partitions each of which has $0$ in its upper left-hand corner when we use the English convention for Young diagrams of a partition. These will be called the  $0$-\textit{corner partitions}.
		\item For each $i>0$ in ascending order, $a_i$ partitions which are $h$-restricted but not $h$-strict.  The residue of a node with row $s$ and column $t$ will be $\widehat{(i+t-s+1)}$, and the partition will be called an $i$-\textit{corner partition}. If $0<i<j\leq n$, the $i$-corner partitions will lie below the $j$-corner partitions.
	\end{itemize}
\end{defn}

\begin{defn} 
	Given a  $i$-corner partition $\lambda$, for a residue $i>0$, and an arbitrary residue $j$,  a $j$-node adjacent to $\lm$ is an \textit{addable $i$-node} of $\lm$ if adding it produces a well-defined partition, and there are also double addable $0$-nodes. The complete set of addable $j$-nodes is the maximal set of $j$-nodes such that adding them produces a well-defined partition. Similarly,  a node of $\lm$ is \textit{$j$-removable} if  removing the node still gives a well-defined partition. The complete set of removable $j$-nodes is the set of all $j$-nodes in $\lambda$ that can be removed and will leave a well-defined partitions.
	
	Now consider a $0$-corner partition,  one of the first $a_0$ partitions in a spin multipartition.  The $j$-addable and removable nodes  can be defined as in the paragraph above, except that the partitions must remain $h$-strict.
\end{defn}

\begin{defn}\label{pm}  For any addable or removable $j$-node on  a partition, we say that a $j$-node is \textit{above} if it is above or to the right, and \textit{below} if it is below or to the left, \cite{LT} and \cite{Fa}.
	
	  In $i$-corner partitions for $i \neq 0$, when $j=0$, it may happen that a removable $0$-node lies directly above an addable $0$-node.  Such a pair of nodes will be designated by $\pm$ and we will vary the method from \cite{LT}. Even though in our definition of ``below'', we would have the addable node ``below'' the removable node, we will in this case regard the removable node as being below.
\end{defn}

Given a positive integral weight $\Lm$, we are interested in generating recursively a set of spin multipartitions that will correspond to a crystal basis in a Kashiwara crystal for $\Lm$. Such a set of spin multipartitions, if they exist, would properly be called $h$-restricted spin multipartitions.  We conjecture that the following algorithm will recursively produce such a set of spin multipartitions. We must first define the $j$-addable and $j$-removable nodes.

\textbf{Spin Multipartition Conjecture} We designate addable and removable $j$-nodes  by $+$ and $-$ as Kleshchev does \cite{Kl}  and write them from left to right as we go up in the multipartition, to get the \textit{$i$-signature}.
Every case of $\pm$ as in Def. \ref{pm} will contribute ``-+''.

	We eliminate any cases of $+-$ from the list.  By standard combinatorial arguments, the order in which the $+-$ pairs are eliminated will not effect the final result.
	
When we have made all the eliminations, we have a string of $-$ corresponding to removable nodes followed by a string of $+$ corresponding to addable nodes.  This list is called the \textit{reduced} $i$-signature. Any or all of the strings could be empty, but if there is a $+$, then the leftmost is the \textit{cogood} node which will be 
added to produce the result of operating by $f_j$ in the Kashiwara crystal and the rightmost will be called the \textit{$j$-good node} for operating by $e_j$. Starting with the empty multipartition, we construct the elements of the crystal basis by divided powers the $f_j$. The resulting multipartions will be called the \textit{$h$-restricted spin multipartitions}.

We conjecture that these $h$-restricted spin multipartitions will label the elements of the Kashiwara crystal generated from the empty multipartition by divided powers.

	\section{THE SERRE RELATIONS FOR $\Lm=\Lm_i$}

We begin with the affine Lie algebra $\mathfrak{g}=A^{(2)}_{2n}$.  Any representation of the affine Lie algebra on a vector space is a representation of the enveloping algebra $U(\mathfrak{g})$. We will eventually want to define a Fock space with a basis of multipartitions. We will start by choosing one of the fundamental weights $\Lm_i$ with $i \neq 0$ and study its Fock space, with a basis given by the appropriate partitions as described above. Eventually, for general highest weight $\Lm$, we will take a tensor product of such single-partition Fock spaces, using the Hopf algebra properties of the quantum enveloping algebra.

We work over the field $\mathbb{Q}(q)$.  For an affine Lie algebra given by a symmetrizable Cartan matrix $A=[a_{jk}]$, the quantum enveloping algebra is generated by $e_j$, $f_j$, $K_j$, $K_j^{-1}$, $D$, subject to relations depending on the Cartan matrix \cite[\S 3]{KMOY}. Set $\ell_{jk}=1-\la h_j,\a_k \ra$.

For each residue, we have a power of $q$ depending on the length of the corresponding root in the symmetric power. In our case this gives
\[
q_i=\begin{cases}
	q, &i=0\\q^2, &0< i < n \\q^4,&i=n.\\
\end{cases}
\]
To define the divided powers,
we set 
\[
[k]_{q_j}=\frac{q_j^k-q_j^{-k}}{q_j-q_j^{-1}}, [k]_{q_j}!=[k]_{q_j}\cdot\cdot\cdot[2]_{q_j}[1]_{q_j},
e_j^{(k)}=\frac{e_j^k}{[k]_{q_j}!},f_j^{(k)}=\frac{f_j^k}{[k]_{q_j}!}.
\]

We give the relations in terms of these powers of $q$.

\begin{align}
	K_jK_k=K_kK_j;&	K_jK_k^{-1}=K_k^{-1}K_j.\\
	K_je_kK_j^{-1}=q_j^{\la h_i,\a_k \ra}e_k;&	K_jf_kK_j^{-1}=q_j^{-\la h_j,\a_k \ra}f_k.\\
	[e_j,f_k]=&\delta_{jk}\frac{K_j-K_j^{-1}}{q_j-q_j^{-1}};\\
	\sum_{t=0}^{\ell_{jk}} (-1)^te_j^{(t)}e_ke_j^{(\ell-t)}=&0,j \neq k.\\
	\sum_{t=0}^{\ell_{jk}} (-1)^tf_j^{(t)}f_kf_j^{(\ell-t)}=&0,j \neq k.		
\end{align}

 We will refer to the last two relations as the quantized Serre relations. The divided powers that appear in the last two equations are defined as in \S2. 

Leclerc and Thibon \cite{LT} gave operations of $f_i$ and $e_i$ on the Fock space of $\Lm=\Lm_0$, which we denote by $\mathcal{F}_0$. Fayers extended these operations to the $h_i$ and the scaling element \cite{Fa}. The vector space $\mathcal{F}_0$ had a basis of $h$-strict Young diagrams, necessarily beginning with residue $0$. We now  let  $\mathcal{F}_i$ represent $i$-corner partitions when $i>0$. We wish to define the action of the Chevalley generators of the  $\mathcal{U}_q(A^{(2)}_{2n})$ on the Fock space $\mathcal{F}_i$.

For any node $x$, we let $A^+_i$ and $R^+_i$ represent the addable and, respectively, removable $j$-nodes above $x$.  Similarly, we let $A^-_i$ and $R^-_i$ represent the addable and, respectively, removable nodes below $x$. For $j=0$, we will have to make a slight change in the case of a $\pm$-pair, in that we will consider the removable $0$-node in a $\pm$-pair to be below the addable node.

We define
\[
N^+_j(x)= \lvert R^+_j(x) \rvert -\lvert A^+_j(x)\rvert; 
N^-_j(x)=\lvert A^-_j(x)\rvert - \lvert R^-_j(x) \rvert.
\]
When a partition $\mu$ is obtained from $\lm$ by adding or removing a $j$-node, we will identify $j$-nodes in $\lm$ and $\mu$ which are unaffected by the change. When there might be a question as to the partition in whose signature the node appears, we will expand the notation and write

\[
N^+_j(x,\mu)=  \lvert R^+_j(x,\mu) \rvert -\lvert  A^+_j(x,\mu)\rvert; 
N^-_j(x,\mu)=\lvert  A^-_j(x,\mu)\rvert - \lvert  R^-_j(x,\mu) \rvert.
\]
To define the action of the generators of $\mathcal{U}$ on $\mathcal{F}_i$, we consider any Young diagram in the natural basis of the Fock space.  We use the action given in \cite[p. 618]{AM}. Let 
$A_j(\lm)$ and $R_j(\lm)$ be the set of addable and removable nodes of $\lm$. and write $\lm^x$ for the result of adding an addable node $x$, $\lm_y$ for the result of removing a removable node $y$. Let $N^d(\lm)$ be the $0$-content of $\lm$, the number of zero residues in the partition. We set
\[
N_j(\lm)=\lvert  A_j(\lm)\rvert  - \lvert  R_j(\lm)\rvert.
\]

For $j>0$, set
\[
e_j(\lm)=\sum_{y \in R_j(\lm)} q_j^{N^+_j(y)}\lm_y.
\]
\[
f_j(\lm)=\sum_{x \in A_j(\lm)} q_j^{N^-_j(x)}\lm^x.
\]	

When $j=0$, define $g(q)=1+q^2$.  Whenever $y \in R(\lm)$ is a removable $0$-node not adjacent to an addable $0$-node or $x \in A(\lm)$ is an addable $0$-node adjacent to a removable $0$-node, we multiply the coefficient by $g(q)$. The formula in \cite{LT} is more complicated because of the possibility of multiple rows beginning with $0$ and ending in $0$, whose length is a multiple of $h$. For our partitions this does not occur.

We finish the definition by setting 
\[
K_{h_j}(\lm) = q_j^{N_j(\lm)}\lm, D(\lm)=q_j^{N^d(\lm)}\lm.
\]
We abbreviate $K_{h_j}$ by $K_j$.

 The action of $D$ commutes with the action of the $K_i$. Both are multiplication by a scalar, and since the action of each leaves the content fixed, neither action changes the other scalar.

\begin{prop}\label{first3}
	The Fock space $\mathcal{F}_i$ for $i \neq 0$ satisfies the relations 5.1-5.3 required of a $U_q(\mathfrak{g})$-module. 
\end{prop}
\begin{proof}

	We must check that applying both sides of a relation to a weight vector produces the same result. The weight vector is $\lm$ a partition in the basis of the Fock space $\mathcal{F}_i$. Because we are using the quantum version of the enveloping algebra and because the residues increase and decrease, the proof is considerably	 more complicated than the corresponding proof in \cite[p.805]{A1}. We must verify all the relations listed above.
	
	\begin{enumerate}
		\item \textbf{Relations}	$K_jK_k=K_kK_j;	K_jK_k^{-1}=K_k^{-1}K_j$.	
	    Since the action of $K_j$ or $K_j^{-1}$ is multiplication by an element of the coefficient ring, these actions commute.
    	\item \label{two} \textbf{Relations} $ K_je_kK_j^{-1}=q_j^{\la h_i,\a_k \ra}e_k;	K_jf_kK_j^{-1}=q_j^{-\la h_j,\a_k \ra}f_k$. We will check this relation for the $f_k$, the case for $e_k$ being dual. Let $\lm$ be a partition.
    	\begin{align*}
    	K_jf_kK_j^{-1}(\lm)=&K_jf_k(q^{-N_j(\lm)}\lm)\\
    	=&q_j^{-N_j(\lm)}K_jf_k(\lm)\\
    	=&q_j^{-N_j(\lm)}K_j\left( \sum_{x \in A_k(\lm)} q_k^{N^-_k(x)}\lm^x  \right)\\
    	=&q_j^{-N_j(\lm)}\left( \sum_{x \in A_k(\lm)} q_k^{N^-_k(x)}q_j^{N_j(\lm^x)}\lm^x  \right)\\ 
    	=&\sum_{x \in A(\lm)}q_j^{-N_j(\lm)+N_j(\lm^x)} q_k^{N^-_k(x)}\lm^x.  		
    	\end{align*} 
    \begin{itemize}
    	\item Case $j=k$. Then $\lm^x$ has one fewer addable $j$-nodes and one more removable $j$-node, so that $-N_j(\lm)+N_j(\lm^x)=-2$, which is precisely $-\la h_j,\a_j \ra$, so $	K_jf_jK_j^{-1}(\lm)=q_j^{-\la h_j,\a_j \ra}f_j(\lm)$.
    	\item $\mid j-k \mid \geq 2$. Adding a $k$-node cannot affect the number of addable and removable $j$-nodes because they are non-adjacent, so
    	$-N_j(\lm)+N_j(\lm^x)=0$, as required. 
    	\item $\mid j-k \mid=1$. In almost all cases, the addition of a $k$-node will add an addable $j$-node or remove a removable $j$ node, but not both.  Thus $-N_j(\lm)+N_j(\lm^x)=1$, which is precisely $-\la h_j,\a_j \ra$. 
    	
    	There are two exceptional cases, corresponding to the off-diagonal entries $-2$ in the Cartan matrix.  In both of those cases, we need to show that we get $-N_0(\lm)+N_0(\lm^x)=2$, which is $-\la h_j,\a_j \ra$. We deal first with  $j=0,k=1$ We list the configurations for adding a $1$-node to a column and leave it to the reader to check that in each case we get  $-N_0(\lm)+N_0(\lm^x)=2$. In each possible configuration, we remove two removable $0$-nodes, add two addable $0$-nodes, or do one of each. The case of adding a $1$-node to a row is dual. A `-' is filled in to the edge, and a `.' is optional. In the first configuration, we add a $1$-node to a $2$-node and create two addable $0$-nodes, in the second configuration we block a removable node and add an addable node, and in the third, we add a $1$-node at the end of a string of two removable $0$-nodes and ensure that they are no longer removable. In the fifth case, we block a removable node and add a new addable $0$-node is to the right of the `+', under an existing addable $0$-node.
    	
    	\young(--\cdot\cdot,-2,-+,-,-,.),
    	\young(---\cdot,-12,-0+,-0,.,.),\young(----\cdot,-012,-00+,.,.,.),
    	
    	\young(--\cdot\cdot,-0,-0,-+,.,.),
   		\young(----,-01\cdot,-0,-+,.,.),\young(----,-01\cdot,-00,-+,.,.)
    
    	Now we turn to the case of $j=n-1,k=n$. In our Young diagrams, we will use $j$ instead of $n-1$ for compactness. Again, adding the $n$-node will remove two removable nodes, add  two addable nodes, or do one of each. 
    	
    	\young(---\cdot,-nj,-j+,\cdot,.),\young(----\cdot,-nj-\cdot,-j+,--,\cdot\cdot),\young(----\cdot,-nj,-j+,--,\cdot\cdot)
    	
    	 The first two configurations are self-dual. The third has a dual which is also a valid configuration.
    \end{itemize}    
    	\item \textbf{Relations} $	[e_j,f_k]=\delta_{jk}\frac{K_j-K_j^{-1}}{q_j-q_j^{-1}}$.
	  We will apply the relation to a partition $\lm$.  In the process of computing $e_jf_k(\lm)$, we created a number of new partitions $\nu=(\lm^x)_y$ with coefficients $q_k^{N_k^-(x,\lm)}q_j^{N_j^+(y,\lm^x)}$ and while computing $f_ke_j(\lm)$, we create partitions $\nu'=(\lm_{y'})^{x'}$ with coefficients $q_j^{N_j^+(y',\lm)}q_k^{N_k^-(x',\lm_{y'})}$ .  To show that the relations hold, these must almost always cancel each other out, certainly when $j \neq k$, but even when $j=k$ if the new partition is not equal to $\lm$.

	   We have several cases:
		\begin{itemize}
			\item \label{2} $\mid j-k \mid \geq 2$. Let  $\nu=(\lm^x)_y$.
		 Because $j$ and $k$ are at a distan fseparations of at least $2$, the nodes  $x,y$ are not adjacent because they are separated by a diagonal.  Thus $y \in R_j(\lm^x)$ can be identified with $y \in R_j(\lm)$, and similarly, $x \in A_k(\lm_y)$ can be identified with $x \in A_j(\lm)$. Adding $x$ or removing $y$ does not affect the addability or removability of the other nodes{\tiny } so the operations are independent, with the same $q$-power coefficients of $e_jf_k$ and $f_ke_j$ so  $\nu$ is canceled by  $\nu'=(\lm_y)^x$ and with the same coefficient.   Thus $(e_jf_k-f_ke_j)\lambda=0$.
			
				\item $\mid j-k \mid =1$. First of all, we may assume $x,y$ are not adjacent, for if they were, then $y$ would not be a removable node of $\lm^x$ and $x$ would not be an addable node in $\lm_y$.   Since they are not adjacent,  adding $x$ does not affect the removability of $y$ and vice versa, so  $\nu=(\lm^x)_y$ and  $\nu'=(\lm_y)^x$ can be identified. The problem is to assure that $\nu$ and $\nu'$ have the same coefficient. If the addable $k$-node $x$ occurs before the removable $j$-node $y$, then neither exponent is changed, i.e. $N_k^-(x,\lm)=N_k^-(x,\lm_y)$,
				because both depend depends on nodes  below $x$ and similarly $N_k^+(y,\lm)=N_k^+(y,\lm^x)$   depends on nodes above $y$. 
				
				If the addable node occurs after the removable node, then both coefficients change, and our problem is to show that they change by the same factor. When $1 \leq j,k \leq n-1$, if we act first by $f_k$, adding a $k$-node either adds an addable $j$-node to the $j$-signature or removes a removable $j$-node, so $N_j^+(y,\lm^x)=N_j^+(y,\lm)-1$. If we act first  by $e_j$, then the removal of a $j$-node before $x$ either removes an addable $k$-node or creates a removable $k$-node, which  means that  $N_k^-(x,\lm_y)=N_k^-(x,\lm)-1$. Since for this range of residues $q_j=q_k$,  the total change of coefficient is the same.
				
				Of the 
				  two special cases, we start with  the pair $0,1$, where  $q_0=q$, $q_1=q^2$. Without loss of generality, we can assume the $j=0, k=1$ since interchanging them would just multiply the result by $-1$. As we showed in (\ref{two}), adding a $1$-node will change the $0$-signature by reversing two nodes, so that $N_0^+(y,\lm^x)=N_0^+(y,\lm)-2$, and the coefficient is divided by $q_0^2$. When we remove $y$, we get a change of only $1$ in the $1$-signature,  $N_1^-(x,\lm_y)=N_1^-(x,\lm)-1$, so the coefficent is divided by $q_1$. Since
				  $q_1=q^2=q_0^2$, the coefficients are equal.  The case of $n-1,n$ works the same, since adding an $n$-node reduces the $n-1$ signature by $2$ and $q_n = q_{n-1}^2$.

			\item $j=k$, $j \neq 0$. First of all, if we add an addable $j$-node $x$ and remove a different removable $j$-node $y$, then they cannot be adjacent, and the same argument used in the case $\lvert j-k\rvert=1$ will show that the coefficients of $(\lm^x)_y$ and $(\lm_y)^x$ are the same. Thus our concern is with the case where we are adding an addable node and then removing it or removing a removable node and adding it back.
			
			 Suppose the partition  $\lm$ has $a$ removable $j$-nodes and $b$ addable $j$-nodes.  This means that $K_j(\lm)=q_j^{b-a}$, so we need to show that  $[e_j,f_j]	(\lm)=\frac{K_j(\lm)-K_j^{-1}(\lm)}{q_j-q_j^{-1}}	=[b-a]_{q_j}(\lm)$.
			
			We will assume that there are more addable nodes than removable nodes, $b>a$. The opposite case is dual with $f_j,e_j$ interchanged. We partition the raw signature into alternating intervals of positivity, containing only $`+'$ signs and intervals of equality, containing equal numbers of plus and minus signs. There may be many ways to choose these intervals but that will not affect the proof.
			
			We perform induction on $a$, starting with the case $a=0$.  Since, in this case,  there are no $j$-removable nodes, $e_j(\lm)=0$, and we need only show that $e_jf_j(\lm)=[b]_{q_j}\lm$. Under $f_j$, the coefficient of the partition $\lm^x$ is $q_j^r$, where $r$ is the number of plus signs before $x$. The number of addable nodes after $x$ is then $b-r-1$, and these enter as negative in calculating the coefficient for $e_j$, so we multiply by $q_j^{r-b+1}$.  Combining these, the total coefficient is $q_j^{2r-b+1}$, where $r$ ranges from $0$ to $b-1$. Thus the total coefficient of $\lm$ is $[b]_{q_j}$, as required.
			
			So we now take as our induction hypothesis that the relations hold for all $b$ and for all integers up to $a$, in such a way that the copies of $\lm$ coming from the equality intervals all cancel out, and the copies of $\lm$ coming from the positivity intervals give the required coefficient. Given a signature with $a+1$ removable nodes and $b+1$ addable nodes, we choose an adjacent `+-' or `-+' in one of the equality intervals, denoting the addable node by $x$ and the removable node by $y$. Start with the case `+-'. If we act first by $f_j$, we multiply $\lm^x$ by $q_j^{N^-_j(x)}$ and then returning to $\lm$ by removing $x$ again, we multiply by $q_j^{N^+_j(y)+1}$ because we have a new removable node to count. If we start with $e_j$, then its coefficient is  $q_j^{N^-_j(y)}$ but when we act by $f_j$, we have an extra addable node, givng $q_j^{N^+_j(x)+1}$.  If our original pair was
			 `-+', we would have to subtract one instead of adding. In either case, the final result is two copies of $\lm$ with coefficients that cancel out.
			\item $j = k = 0$: Once again, the interesting case is where we either remove a $0$-node and add it back or add a $0$-node and remove it again.  In either case we get $\lm$, and we must calculate the coefficient of $\lm$.  As before, we will do the case $b>a$, the other case being dual, and we will do it by induction on $a$. However, this time the case  is more complicated because, when we add the second of two adjacent $0$-nodes, we multiply the coefficient by $g(q)=1+q^2$.
			
			The configurations given above in (\ref{two}) were intended to give all possible configurations for adding a $1$-node, but if we either remove the $+$ or replace it with $1$, we get all possible configurations for addable or removable $0$-nodes, since every $0$-node is adjacent to a $1$-node. Adding or removing one of these nodes does not create new addable or removable nodes. From examining the configurations, we see that $0$-nodes come in pairs, the possibilites being $++$, $\pm$, $-+$, and $--$. We have three possibilities, which we will represent by $2,1,0$, according to the number of removable nodes, occurring in a string of length $(b+a)/2$. We will call this string the doubled signature.
			
			When $a=0$, we have only $0$ in the doubled signature, and the results for adding and then removing an addable node are just as in the case $j \neq 0$. Now, assuming $b>a$, we divide the doubled signature into positivity intervals containing only $0$, or equality invervals containing only $1$'s and equal numbers of $0$ and $2$. Since an equality interval does not affect $N_0^-(x)$ or $N_0^+(y)$.  First suppose that one of the equality intervals contains a $1$, which represents a $-+$ or a $\pm$. If it were omitted, then by our induction hypothesis, all the copies of $\lm$ coming from the equality intervals would cancel out. Furthermore, $N_0^-(x)$ or $N_0^+(y)$ for all other $x$ and $y$ would be unaffected. We cannot operate by either $e_0f_0$ or $f_0e_0$  if we take $x$ to be the $+$ and $y$ to be the minus. 
			
			Suppose we take $x$ to be addable and then removable. Then the exponent of $q_0$ in the coefficient will be $N_0^-(x)(1+q_0^2)$ for $\lm$ and then $N_0^-(x)$ for $\lm^x$, which is $N_0^-(y)+1$ for $\lm$, because $y$ is also counted as a removable $0$ node.  In the opposite direction, if we first remove and then add back $y$, we get $N_0^+(y)(1+q_0^2)$ for removing $y$ and then
			$N^-(y)+1$ for $\lm_y$.
			
			Finally, we consider the actions from an adjacent $02$ or $20$.  As always, adding an addable node from the pair and removing another node elsewhere will cancel out removing and then adding those same nodes. Our problem will be removing and adding the same node $y$ or adding and removing $x$.  We take the nodes lying on the boundary.  Start with $02$, letting $x$ be the leftmost addable node, and $y$ the rightmost removable . We multiply by $q_0$ to the power $N_0^-(x, \lm)$. Between $x$ and $y$, we have an addable and a removable $0$-node, which cancel out, so that $N_0^+(x, \lm^x)=N^+(y,\lm)+1$, the added $1$ being for $y$, and we 		
			mulltiply by $q_0$ to the power $an_0^+(x,\lm^x)+1$.  On the other hand, if we remove $y$ and then restore it, the exponent is $N_0^+(y)$ and the second is $N_0^-(y,\lm_y)$, which equals $N_0^-(x, \lm)+1$.
		\end{itemize}
	\end{enumerate}
\end{proof}

In the Appendix B to the paper on perfect crystals by Kashiwara, Miwa, Petersen and Yung, \cite{KMPY}, the authors proved that for any symmetrizable Kac-Moody Lie algebra, if one can prove that an integrable module satisfies the relations 5.1-5.3, then it satisfies the remaining relations.  To make this paper self-contained, we have proven the quantized Serre relations, or, at least, one of the dual pair.
We split the proof into two cases, when $a_{jk}=-1$ and the more complicated case when $a_{jk}=-2$.

We first prove a lemma.

\begin{lem}\label{divided}
	For $\Lm = \Lm_i$, $i \neq 0$, let $\mu$ be a partition.
	\begin{enumerate}
		\item If $\mu$ has addable, non-adjacent, $j$-nodes $x,y$ with $x$ before $y$, let $R_{x,y}=q_j^{N_j^-(x,\mu)+N_j^-(y,\mu)}G_{x,y}(q)$, where if $j \neq 0$, $G_{x,y}(q)=1$ and if $j=0$, then $G_{x,y}(q)$ is the product of as many copies of $g(x)=1+q_j^2$ as there are addable nodes in the set $x,y,z$ which is the second $0$ in a pair.
		Then in $f_j^{(2)}(\mu)$, the coefficient of $\mu^{x,y}$ is 
		\[
		q_j^{-1}R_{x,y}.
		\]
				\item If $\mu$ has three addable non-adjacent $j$-nodes $x,y,z$ with $x$ before $y$ and $y$ before $z$, define 
			$R_{x,y,z}=q_j^{N_j^-(x,\mu)+N_j^-(y,\mu)+N_j^-(z,\mu)}G_{x,y,z}(q)$	 with $G_{x,y,z}(q)$ defined as above.			
				 Then in $f_j^{(3)}(\mu)$, the coefficient of $\mu^{x,y,z}$ is 
		\[
		q_j^{-3}R_{x,y,z}.
		\]
	\end{enumerate}
\end{lem}
\begin{proof}
	\begin{enumerate}
		\item If an addable node $x_1$ is before an addable node $x_2$ then $N_j^-(x_2,\mu^{x_1})=N_j^-(x_2,\mu)-2$ because the addable node which contributed $+1$ to $N_j^-(x_2,\mu)$ is now removable and contributes $-1$ instead. On the other hand, $N_j^-(x_1,\mu^{x_2})=N_j^-(x_1,\mu)$ because nodes after $x_1$ to not contribute to the signature. 
		\begin{align*}
		 f_j^{2}(\mu)=&q_j^{N_j^-(x,\mu)+N_j^-(y, \mu^x)}G_{x,y}(q)(\mu^x)^y+q_j^{N_j^-(y,\mu)+N_j^-(x, \mu^y)}G_{x,y}(q)(\mu^y)^x+\dots\\
		 =&q_j^{N_j^-(x,\mu)+N_j^-(y, \mu)}G_{x,y}(q)(q_j^{-2}+1)\mu^{x,y}+\dots.
		 \end{align*}
	 To get the divided power we divide by $q_j+q_j^{-1}$ and get the desired result.
	 \item As in $(1)$, in calculating 
	 $N_j^-(x_2, \lm^{x_1,x_3})$, addable nodes after $x_2$ can simply be erased from $\mu$ but each addable node $x_1$ before $x_2$ with contribute $-2$ to the exponent. In the case $j=0$, the extra factors of $1+q^2$ which are included when adding a second zero to an adjacent zero are included in $Q$.
	 \begin{align*}
	 	f_j^{3}(\mu)=&q_j^{N_j^-(x,\mu)+N_j^-(y, \mu^x)+N_j^-(z, (\mu^x)^y)}G_{x,y,z}(q)((\mu^x)^y)^z+\dots\\
	 		 	+&q_j^{N_j^-(y,\mu)+N_j^-(x, \mu^y)+N_j^-(z, (\mu^y)^x)}G_{x,y,z}(q)((\mu^y)^x)^z+\dots\\
	 		 	+&\dots\\
	 		 	+&q_j^{N_j^-(z,\mu)+N_j^-(y, \mu^z)+N_j^-(x, (\mu^z)^y)}G_{x,y,z}(q)((\mu^z)^y)^z+\dots\\	 		 	
	 		 	 	=&(q_j^{-6}+2q_j^{-4}+2q_j^{-2}+1)R_{x,y,z}\mu^{x,y,z}+\dots.
	 \end{align*}
	 To get the divided power we divide by $[2]_{q_j}[3]_{q_j}=q_j^{-3}+2q_j^{-1}+2q_j+q_j^{3}$ and get the desired result.
	\end{enumerate}
\end{proof}

\begin{prop}\label{S.2}
	When $\ell=1-a_{jk} \leq 2$, the quantized Serre relations hold.
\end{prop}

\begin{proof}
		   When $a_{jk}=0$, this asserts that they commute, Since this happens when $\lvert j-k\rvert \geq 2$, the operations are independent, as we argued earlier, thus they do indeed commute.
		
		Now we consider the case that $j,k$ differ by one, where the problems arise. 
		We start with when $j\neq 0$ and $k \neq n$.  In all these cases,  $a_{jk}=-1$. Then $\ell=1-a_{jk} \leq 2$. We will, as usual, do the case of $f_j$, the case of $e_j$ being dual. With $\ell=2$, we are reduced to establishing that
		\[
		f_j^{(2)}f_k+	f_kf_j^{(2)}=f_jf_kf_j.
		\] 
		
			Now take a general partition $\lm$, and let $\lm'$ be one of the partitions obtainable by adding two $j$-nodes and one $k$-node.
	
		  Let $x,y$ be the addable $j$-nodes leading the $\lm'$, with $x$ below $y$. Let $w$ be the addable $k$-node.  We must now have a notation for the position of $w$ with respect to $x$, $y$. 
		
		 If $w$ is adjacent to $x_i$, then we write $x_i \rightarrow w$ if $x_i$ must be added before $w$ can be added, and we write $w \rightarrow x_i$ if $w$ must be added before $x_i$.  Thus in describing the positions of the addable nodes, we must choose an ordering $\sigma$ from among  the following seven possibilities:
		\begin{itemize}
			\item Non-adjacent addable nodes: $\sigma_1=(w,x,y), \sigma_2=( x,w,y), \sigma_3=(x,y,w)$.
			\item Adjacent addable nodes: $\sigma_4=(w \rightarrow x, y),\sigma_5=(x \rightarrow w,y),\sigma_6=(x, w \rightarrow y),\sigma_7= (x, y \rightarrow w).$
		\end{itemize}
		
			We let $A$ denote the path $j,j,k$, we let $B$ denote the path $j,k,j$ and we let $C$ denote the path $k,j,j$. The case of $B$ splits into two cases, $B_1$ when the path starts by adding $x$ and  $B_2$ when the path starts by adding $y$.  If $w \rightarrow x_i$, then $B_i$ will be zero because $x_i$ cannot be added before $w$. If there is some $i$ for which $x_i \rightarrow w$, then $C$ will be zeri.

			The coefficient of $\lm'$ along each path depends on the  total collection of $+$ and $-$, nodes, but when the nodes to be added are not adjacent, we can divide out by the common coefficient
		\[
		Q=Q(\lm, \{x,y,w\}) = q_j^{N_j^-(x,\lm)}q_j^{N_j^-(y,\lm)}q_k^{N_k^-(w,\lm)}G_{x,y,z}(q).
		\]
		
		When one of the addable nodes depends on the addition of another node, we will calculate its original coefficient after the addition of that extra node.
		What remains after dividing by $Q$ depends on the local situation at the three nodes $x, y, w$.  These are the coefficients which we will calculate.
		
		\noindent \textbf{Case 1:} $1 \leq j,k \leq n-1$.  In this case, $q_j=q_k=q^2$.
		\begin{itemize}
			\item $\sigma_1=(w,x,y)$: The concept ``before'' in a spacial sense, pertaining to the signature, was defined for nodes of the same reside, but we can extend it to non-adjacent nodes of residues which differs by $1$, since any node of residue $k$ is adjacent to a node of residue $j$. With this extension, we want $w$ to be before $x$ which is before $y$. We let $\lm$ be any partition that has three addable nodes in the given configuration, and let $\lm_1$ be the unique partition obtained by adding the $k$-node $w$ and the two $j$-nodes $x,y$. In calculating exponents, we can automatically discard any added nodes which are after the currently added node spacially, so that in the calculation below for $A$, we substitute $N_k^-(w,\lm)$ for $N_k^-(w,\lm^{x,y})$ because $x$ and $y$ were assumed to be after $w$ in our extended version of before and after.
\begin{enumerate}
	\item For operation $A$, we apply Lemma \ref{divided} (1).
	\begin{align*}
		f_kf_j^{2}(\lm)=&f_k(q_j^{-1}q_j^{N_j^-(x,\lm)+N_j^-(y,\lm)} \lm^{x,y})+\dots\\
			=&q_k^{N_k^-(w,\lm^{x,y})}q_j^{N_j^-(x,\lm)
			+N_j^-(y,\lm)}q_j^{-1}\lm_1+\dots\\
			=&q_k^{N_k^-(w,\lm)}q_j^{N_j^-(x,\lm)
			+N_j^-(y,\lm)}q_j^{-1}\lm_1+\dots\\
		=&Qq_j^{-1}\lm_1+\dots.\\
	\end{align*}
	\item Here we will use $N_k^-(w,\lm)=N_k^-(w,\lm^x)=N_k^-(w,\lm^y)$  because $x$ and $y$ were assumed to be after $w$. We will also need $N_j^-(y,\lm^{w,x})=N_j^-(y,\lm)+1-2$ and $N_j^-(x,\lm^{w,y})=N_j^-(x,\lm)+1$. The $1$ comes in each case because the addition of $w$ either blocks a removable $j$-node or adds an addable $j$-node.  The $-2$ in the case of $y$ comes because the preceding addable $j$-node $x$ became a removable $j$-node after $x$ was added.
	\begin{align*}
		f_jf_kf_j(\lm)=&f_j(f_k(q_j^{N_j^-(x,\lm)}\lm^x+q_j^{N_j^-(y,\lm)}\lm^y)))+\dots\\
		=&f_j(q_k^{N_k^-(w,\lm)}(q_j^{N_j^-(x,\lm)}\lm^{w,x}+q_j^{N_j^-(y,\lm)}\lm^{w,y}))+\dots.\\
		=&q_k^{N_k^-(w,\lm)}(q_j^{N_j^-(x,\lm)+N_j^-(y,\lm^{w,x})}\\	+&q_j^{N_j^-(y,\lm)+N_j^-(x,\lm^{w,y})})\lm_1+\dots\\
		=&Q(q_j^{-1}+q_j)\lm_1+\dots.
	\end{align*}
\item We again apply Lemma \ref{divided} (1). Now we also use that 
$N_j^-(x,\lm^{w})=N_j^-(x,\lm)+1$ and $N_j^-(y,\lm^{w})=N_j^-(y,\lm)+1$.
	\begin{align*}
	f_j^{2}f_k(\lm)=&f_j(q_k^{N_k^-(w,\lm)} \lm^{w})+\dots\\
	=&(q_k^{N_k^-(w,\lm)} q_j^{-1}q_j^{N_j^-(x,\lm^w)+N_j^-(y,\lm^w)} \lm^{w,x,y})+\dots.\\
	=&q_k^{N_k^-(w,\lm)}q_j^{N_j^-(x,\lm)
		+N_j^-(y,\lm)}q_j\lm_1+\dots\\
	=&Qq_j\lm_1+\dots.\\
\end{align*}

\end{enumerate}			
			
			This is summarized in the first row of the table below.  In the subsequent $\sigma_i$, will will not go into the same detail.  The coefficient $Q$ depends on the choice of $\lm$, but the change in the coordinates for $A,B,C$ depends only on $\sigma_i$, which gives the spacial ordering of the addable nodes, while $A,B.C$ give the temporal ordering.		
				\item $\sigma_2=( x,w,y)$: The partition $\lm$ is any partition that has nodes $x,w,y$ in that spacial ordering and $\lm_2$ is the unique partition obtained by adding all three addable nodes to $\lm$. The addition of $x$ before $w$ will multiply the coefficient of $\lm_2$ by $q_k$, since it will either remove a removable $k$-node or add an addable $k$-node. This will occur in $A$ and in $B_1$. The addition of $w$ before $y$ will multiply by $q_j$, and this will happen in $B_2$ and $C$.
			\item $\sigma_3=(x,y,w)$: From Lemma \ref{divided}, we get $q_j^{-1}q_j^{N_j^-(x,\lm)+N_j^-(y,\lm)}q_k^{N_k^-(w,\lm^{x,y})}$ in $A$. Th two nodes $x, y$ before $w$ will give $N_k^-(w, (\lm^x)^y)=N_k^-(w,\lm)+2$ so the coefficient of change is $q_k^2q_j^{-1}$. In $C$, the $w$ is added before these and so the $k$-signature is unchanged, and also $N_j^-(x,\lm^w)=N_j^-(x,\lm)$ $N_j^-(y,\lm^w)=N_j^-(y,\lm)$ since $w$ is after $x,y$ spacially.  Now consider $B_1$.  The node $x$ added before $w$ will multiply the coefficient by $q_k$ since it will either block a removable $k$ node or add an addable $k$-node, $N_k^-(w, \lm^x)=N_k^-(w,\lm)+1$.  However, when we add $y$, we multiply by $q_j^{-2}$ because the addable node at $x$ turned to a removable node $x$, and $w$, being above $y$, does not make any change, giving $q_kq_j^{-2}$. In $B_2$, on the contrary, we add $y$ while $x$ is still addable, and get a coefficient $q_k$. When we substitute $q_k=q_j$, we get that the coefficients associated to $A$ and $B_2$ cancel those from $C$ and $B_1$.
			
			\item In the remaining four orderings of the addable nodes, we expect that some of the operatiosn will be impossible because of the precedence. We start with 	$\sigma_4=(w \rightarrow x, y)$.	The coefficients for $A$ and for $B_1$ are zero.  For $C$, the coefficient of change from the addition of $w$ is $1$, and similarly, for $f_j^{(2)} $, we get $1$ after division. For $B_2$, adding $y$ is also $1$, $x$ is not yet an addable node, The coefficients of change both for $w$ and afterwards for $x$ are also $1$. 
			
			\item For $\sigma_5=(x \rightarrow w, y)$, we have $C$ and $B_2$ equal to $0$. The operation $f_j^{(2)}$ in $A$ gives $1$, 	and there is no change in the number of addable or removable nodes before $w$, so the coefficient there is also $1$. 	
			
			\item For $\sigma_6 = (x,w \rightarrow y$), both $A$ and $B_2$ are zero, since there is no possiblity to add $y$ before $w$. In $B_1$, the addition of $x$ multiplies the coefficient for adding $w$ by $q_k$, and then the new removable $j$-node $x$ multiplies by $q_j^{-1}$. In $C$, the addition of $w$ makes no change in the coefficient and the divided power also give $1$. 
			
			\item For $\sigma_7 = (x,y \rightarrow w$), $B_1$ and $C$ are $0$. The addition of $x$ before $w$ in $A$ multiplies the coefficient by $q_k$. In $B_2$, the addable node at $x$ is already part of $Q$. as is the coefficient for adding $w$.  None of the new nodes is before $x$, so the total change coefficent is $1$.
			
		\end{itemize}

					\begin{center}
				\begin{tabular}{ | l || c | c | c | c |  }
					\hline
					$\sigma$ &$A$ & $B_1$&$B_2$  & $C$ \\ \hline \hline
					$\sigma_1$ &$q_j^{-1}$&$q_j^{-1}$&$q_j$&$q_j$\\ \hline
					$\sigma_2$ &$q_kq_j^{-1}$&$q_kq_j^{-1}$&$1$&$1$\\ \hline
					$\sigma_3$ &$q_k^2q_j^{-1}$&$q_kq_j^{-2
					}$&$q_k$&$q_j^{-1}$\\\hline
					$\sigma_4$ &$0$&$0$&$1$&$1$\\ \hline
					$\sigma_5$ &$q_j^{-1}$&$q_j^{-1}$&$0$&$0$\\ \hline
					$\sigma_6$ &$0$&$q_kq_j^{-2}$&$0$&$q_j^{-1}$\\ \hline
					$\sigma_7$ &$q_kq_j^{-1}$&$0$&$1$&$0$\\ \hline \hline					
					\hline
				\end{tabular}
			\end{center}

 The sum of the coefficients for $A$ and $C$ equals the sum of the coefficients for $B_1$ and $B_2$.

\noindent \textbf{Case 2:} $j=1, k=0$ or $j=n, k=n-1$.  In both of these cases, $a_{jk}=-1$, so we have the same equation of length $3$.	Now, however, $q_j=q_k^2$. As we explained in (\ref{two}), when we add a $1$-node, the coefficient coming from the $0$-signature changes by $2$.  Similarly, when we add an $n$-node, the coefficient in the $n-1$-signature changes by $2$. Thus we get the same table as in the previous case, except that the powers of $q_k$ are multiplied by $2$, as follows:
		
\begin{center}
	\begin{tabular}{ | l || c | c | c | c |  }
		\hline
		$\sigma$ &$A$ & $B_1$&$B_2$  & $C$ \\ \hline \hline
		$\sigma_1$&$q_j^{-1}$&$q_j^{-1}$&$q_j$&$q_j$\\ \hline
		$\sigma_2$ &$q_k^2q_j^{-1}$&$q_k^2q_j^{-1}$&$1$&$1$\\ \hline
		$\sigma_3$ &$q_k^4q_j^{-1}$&$q_k^2q_j^{-2}$&$q_k^2$&$q_j^{-1}$\\\hline
		$\sigma_4$ &$0$&$0$&$1$&$1$\\ \hline
		$\sigma_5$ &$q_j^{-1}$&$q_j^{-1}$&$0$&$0$\\ \hline
		$\sigma_6$ &$0$&$q_k^2q_j^{-2}$&$0$&$q_j^{-1}$\\ \hline
		$\sigma_7$ &$q_k^2q_j^{-1}$&$0$&$1$&$0$\\ \hline \hline					
		\hline
	\end{tabular}
\end{center}
		
As in the previous case, the sum of the coefficients for $A$ and $C$ equals the sum of the coefficients for $B_1$ and $B_2$.
\end{proof}

\begin{examp}
	Let $h=7$, $n=3$, $\Lm=\Lm_3$, $j=1,k=2$.  We take a partition $\lm=(2,2,2,1^3)$, which has a raw $1$-signature $(-,+,+)$ and a raw $2$-signature $(+)$. To $\lm$ we are to add two $1$-nodes $x,y$ and one of two  $2$-nodes $w$ or $z$. There are three ways to do this corresponding to 
	$\sigma_1$ and $\sigma_7$.  We denote the resulting partitions by $\lm_1=(3,2,2,2,1^3)$ and $\lm_7=((3,3,2,2,1^2)$.  The Young diagrams with residues are 
	
	$\lm=$ \young(32,23,12,0,0,1),	$\lm_1=$\young(32y,23,12,0x,0,1,w)=\young(321,23,12,01,0,1,2)	$\lm_7=$\young(32y,23z,12,0x,0,1)=\young(321,232,12,01,0,1).
	
	The main point to this example is to illustrate the coefficient $Q_i$ which is crucial to the proof of Prop. \ref{S.2}, so we will make the calculation in full.  
	\begin{itemize}
		\item $Q_1= q_j^{N_j^-(x,\lm)}q_j^{N_j^-(y,\lm)}q_k^{N_k^-(w,\lm)}=q_j^{-1}q_j^0q_k^0=q_j^{-1}.$
		\item $Q_7= q_j^{N_j^-(x,\lm)}q_j^{N_j^-(y,\lm)}q_k^{N_k^-(z,\lm^y)}=q_j^{-1}q_j^0q_k^0=q_j^{-1}$.
	\end{itemize}

We now calculate  	$f_kf_j^{(2)}(\lm)$, $f_jf_kf_j(\lm)$, and	$f_j^{(2)}f_k(\lm)$, giving the coefficient of $\lm_1$ in the first row of the table below and the coefficient of $\lm_7$ in the second row. As usual, after calculating $f_j^2$' we must divide by $q_j^{-1}+q_j$ to get the divided power.
\begin{align*}
	f_kf_j^{(2)}(\lm)=&q_j^{-1}q_j^{-1+0}f_k(\lm^{x,y})\\
	=&q_j^{-2}(\lm_1+q_k\lm_7).\\
\\
    f_j	f_kf_j(\lm)=&q_j^{-1}f_jf_k(\lm^x)+f_jf_k(\lm^y)\\
	=&q_j^{-1}f_j((\lm^x)^w)+(f_j((\lm^y)^w)+f_j((\lm^y)^z)\\
	=&q_j^{-2}\lm_1+\lm_1+q_j^{-1}\lm_7.\\
\\
	f_j^{(2)}f_k(\lm)=&f_j^{(2)}(\lm^w)\\		=&(q_j^{N(x,\lm^w)+N(y,\lm^{w,x})}+q_j^{N(y,\lm^w)+N(x,\lm^{w,y})})/(q_j^{-1}+q_j)\\
=&(q_j^{0+(-1)}+q_j^{1+0})/(q_j^{-1}+q_j)\lm_1\\
=&\lm_1.\\
\end{align*}

To summarize, we  give the coefficient of $\lm_1$ in the first row of the table below and the coefficient of $\lm_7$ in the second row. 
 For  $f_jf_kf_j(\lm)$ we get a sum of which the first summand comes from adding the node $x$ before $w$ or $z$,  if possible, and otherwise $0$ and the second summand  from adding $y$ before $w$ or $z$. The last column contains the $Q_i$, which we calculated above.  In order to get the corresponding row in the table in the proof, we divide by $Q$.
		
\begin{center}
	\begin{tabular}{ | l || c | c | c | c |  }
		\hline
		$\sigma$ &$f_k	f_j^{(2)}$ & $f_jf_kf_j$&$f_j^{(2)}f_k$  & $Q$ \\ \hline \hline
		$\sigma_1$ &$q_j^{-2}$&$q_j^{-2}+1$&$1$&$q_j^{-1}$\\ \hline \hline				
		$\sigma_7$ &$q_j^{-2}q_k$&$0+q_j^{-1}q_k$&$0$&$q_j^{-1}$\\ \hline \hline					
		\hline
	\end{tabular}
\end{center}

\end{examp}
\begin{prop}\label{S.3}
	When $\ell=1-a_{jk} = 3$, the quantized Serre relations hold.
\end{prop}
\begin{proof}

Finally, we consider $j=0, k=1$ or $j=n-1,k=n$, in which case $a_{jk}=-2$.  The quantized Serre relaitons for $f_j,f_k$ is equivalent to 
\[
f_j^{(3)}+f_jf_kf_j^{(2)}=f_j^{(2)}f_kf_j+f_kf_j^{(3)}.
\]		
As in the case of $a_{jk}=-1$, we will let $x,y,z$ be the addable $j$-nodes, which we assume non-adjacent and $w$ be the addable $k$ node. We start with a partition $\lm$.
	
	The possible paths to the new partitions $\lm'$ can be described as follows: 
	\begin{itemize}
		\item $A=f_kf_j^{(3)}$,
		\item $B_1=f_jf_kf_j^{(2)}$, with $x$ added last.
		\item $B_2=f_jf_kf_j^{(2)}$, with $y$ added last.
		\item $B_3=f_jf_kf_j^{(2)}$, with $z$ added last.
		\item $C_1=f_j^{(2)}f_kf_j$, with $x$ added first.
		\item $C_2=f_j^{(2)}f_kf_j$, with $y$ added first.
		\item $C_3=f_j^{(2)}f_kf_j$, with $z$ added first.
		\item $D=f_j^{(3)}f_k$.
		
	In order for the relation to hold, we must have the sum of $A$ and the $C_t$ equal to $D$ and the sum of the $B_t$. As long as the addable nodes are not adjacent and we have no adjacency chains of length three or more, this requires no new techniques, except that we must use the full strength of Lemma \ref{divided}.    
	There will now be ten possible orderings of the addable nodes, which we now list, each with its basic coefficient:
	\begin{itemize}
	\item $\tau_1=(w,x,y,z)$; $Q=	 q_j^{N_j^-(x,\lm)}q_j^{N_j^-(y,\lm)}q_j^{N_j^-(z,\lm)}q_k^{N_k^-(w,\lm)}G_{x,y,z}(q)$
	\item $\tau_2=(x,w,y,z)$; $Q$,
	\item $\tau_3=(x,y,w,z)$; $Q$,
	\item $\tau_3=(x,y,z,w)$; $Q$,
	\item $\tau_5=(w \rightarrow x,y,z)$; $Q_5=q_j^{N_j^-(x,\lm^w)}
	q_j^{N_j^-(y,\lm)}q_j^{N_j^-(z,\lm)}q_k^{N_k^-(w,\lm)}G_{x,y,z}(q)	$
	\item $\tau_6=(x \rightarrow w,y,z)$; $Q_6=q_j^{N_j^-(x,\lm)}q_j^{N_j^-(y,\lm)}q_j^{N_j^-(z,\lm)}q_k^{N_k^-(w,\lm^{x})}G_{x,y,z}(q)$,
	\item $\tau_7=(x,w \rightarrow y,z)$; $Q_7=q_j^{N_j^-(x,\lm)}q_j^{N_j^-(y,\lm^w)}q_j^{N_j^-(z,\lm)}q_k^{N_k^-(w,\lm)}G_{x,y,z}(q)$
	\item $\tau_8=(x,y \rightarrow w,z)$; $Q_8=
	q_j^{N_j^-(x,\lm)}q_j^{N_j^-(y,\lm)}q_j^{N_j^-(z,\lm)}q_k^{N_k^-(w,\lm^{y})}G_{x,y,z}(q)$,
	\item $\tau_9=(x,y,w \rightarrow z)$; $Q_9=
	q_j^{N_j^-(x,\lm)}q_j^{N_j^-(y,\lm)}q_j^{N_j^-(z,\lm^w)}q_k^{N_k^-(w,\lm)}G_{x,y,z}(q)$,
	\item $\tau_{10}=(x,y,z \rightarrow w)$; $Q_{10}=
	q_j^{N_j^-(x,\lm)}q_j^{N_j^-(y,\lm)}q_j^{N_j^-(z,\lm)}q_k^{N_k^-(w,\lm^{z})}G_{x,y,z}(q)$.
	\end{itemize}

We summarize the results for these ten cases in a table, similar to the summary table in the previous proposition. Here again, we have $q_k=q_j^2$ and sometimes it is necessary to use this equality to get the relation. In addition, the addition of a $k$-addable node before a $j$-node will add $2$ rather than $1$ so that The power of $q_j$ will always be odd. We illustrate the difference by doing the calculation for $\tau_1$ in full.

As listed above,  $\tau_1=(w,x,y,z)$.  In this case, 
\[Q=q_j^{N_j^-(x,\lm)}q_j^{N_j^-(y,\lm)}q_j^{N_j^-(z,\lm)}q_k^{N_k^-(w,\lm)}G_{x,y,z}(q).\]
We now calculate the four compositions of $f_j$ and $f_k$ which appear in the quantized Serre relations using the fact that since $w$ is before $x,y,z$ spacially, we have $N_k^-(w,\lm^{x,y,z})=N_k^-(w,\lm)$:

\begin{align*}
	f_kf_j^{(3)}(\lm)=&q_j^{-3}q_j^{N_j^-(x,\lm)+N_j^-(y,\lm)+N_j^-(z,\lm)}f_k(\lm^{x,y,z})\\
	=&q_j^{-3}q_j^{N_j^-(x,\lm)+N_j^-(y,\lm)+N_j^-(z,\lm)}q_k^{N_k(w,\lm^{x,y,z})}\lm_1\\
	=&q_j^{-3}Q\lm_1.\\
\end{align*}
\begin{align*}
		f_jf_kf_j^{2}(\lm)=&q_j^{-1}f_jf_k(q_j^{N_j^-(y,\lm)+N_j^-(z,\lm)} \lm^{y,z)}+q_j^{N_j^-(x,\lm)+N_j^-(z,\lm)}\lm^{x,z}+q_j^{N_j^-(x,\lm)+N_j^-(y,\lm)}\lm^{x,y}\\
		=&q_j^{-1}q_k^{N_k^-(w,\lm)}f_j(q_j^{N_j^-(y,\lm)+N_j^-(z,\lm)}\lm^{y,z}+q_j^{N_j^-(x,\lm)+N_j^-(z,\lm)}\lm^{x,z}+q_j^{N_j^-(x,\lm)+N_j^-(y,\lm)}\lm^{x,y})\\
	=&q_j^{-1}q_k^{N_k^-(w,\lm)}(q_j^{N_j^-(x,\lm^{w,y,z})+N_j^-(y,\lm)+N_j^-(z,\lm)}+q_j^{N_j^-(x,\lm)+N_j^-(y,\lm^{w,x,z})+N_j^-(z,\lm)}\\
	+&q_j^{N_j^-(x,\lm)+N_j^-(y,\lm)+N_j^-(z,\lm)})\lm_1\\
	=&q_j^{-1}Q(q_j^2+q_J^0+q_j^{-2})\lm_1\\
	=&(q_j+q_j^{-1}+q_j^{-3})Q\lm_1.\\
	\\
	f_j^{(2)}f_kf_j(\lm)=&f_j^{(2)}f_k(q_j^{N_j^-(x,\lm)}\lm^x+q_j^{N_j^-(y,\lm)}\lm^y+q_j^{N_j^-(z,\lm)}\lm^z),\\
	=&f_j^{(2)}(q_k^{N_k^-(w,\lm^x)}q_j^{N_j^-(x,\lm)}\lm^{w,x}
		+q_k^{N_k^-(w,\lm^y)}q_j^{N_j^-(y,\lm)}\lm^{w,y}\\
		+&q_k^{N_k^-(w,\lm^z)}q_j^{N_j^-(z,\lm)}\lm^{w,z})\\
	=&q_k^{N_k^-(w,\lm)}q_j^{-1}(q_j^{N_j^-(x,\lm)+N_j^-(y,\lm^{w,x})+N_j^-(z,\lm^{w,x})}+q_j^{N_j^-(x,\lm^{w,y})+N_j^-(y,\lm)+N_j^-(z,\lm^{w,x})}\\
	+&q_j^{N_j^-(x,\lm^{w,y})+N_j^-(y,\lm^{w,z}))+N_j^-(z,\lm)})\lm_1\\
	=&q_j^{-1}Q\left (q_j^{0+(2-2)+(2-2)}+q_j^{2+0+(2-2)}+q_j^{2+2+0)} \right)\lm_1\\
	=&\left(q_j^{-1}+q_j+q_j^3 \right)Q\lm_1.\\
\\
	f_j^{(3)}f_k{\lm}=&q_k^{N_j^-(w,\lm)}f_j^{(3)}(\lm)\\
	=&q_k^{N_j^-(w,\lm)}q_j^{-3}q_j^{N_j^-(x,\lm^w)+N_j^-(y,\lm^w)+N_j^-(z,\lm^w)}\\
	=&q_j^{-3}q_j^{2+2+2}Q\lm_{1}\\
	=&q_j^{3}Q\lm_{1}.\\	
\end{align*}
\begin{center}
	\begin{tabular}{ | l || c | c | c | c | c | c | c | c | }
		\hline
		$\tau$ &$A$ & $B_1$&$B_2$&$B_3$ & $C_1$&$C_2$&$C_3$&$D$ \\ \hline \hline
		$\tau_1$&$q_j^{-3}$ &$q_j$&$q_j^{-1}$&$q_j^{-3}$&$q_j^{-1}$&$q_j$&$q_j^{3}$&$q_j^{3}$\\ \hline
		$\tau_2$ &$q_j^{-3}q_k^2$&$q_j^{-1}$&$q_j^{-1}q_k^2$&$q_j^{-3}q_k^2$&$q_j^{-1}q_k^2$&$q_j^{-1}$&$q_j$&$q_j$\\ \hline
		$\tau_3$ &$q_j^{-3}q_k^2$&$q_j^{-1}q_k$&$q_j^{-1}q_k$&$q_j^{-3}q_k^2$&$q_j^{-1}q_k$&$q_j^{-1}q_k$&$q_j^{-1}$&$q_j^{-1}$\\\hline
		$\tau_4$ &$q_j^{-3}q_k^3$&$q_j^{-1}q_k^2$&$q_j^{-3}q_k^2$&$q_j^{-5}q_k^2$&$q_j^{-5}q_k$&$q_j^{-3}q_k$&$q_j^{-1}q_k$&$q_j^{-3}$\\ \hline
		$\tau_5$ &$0$&$q_j^{-1}$&$0$&$0$&$0$&$q_j^{-1}$&$q_j$&$q_j$\\ \hline
		$\tau_6$ &$q_j^{-3}$&$0$&$q_j^{-3}$&$q_j^{-1}$&$q_j^{-1}$&$0$&$0$&$0$\\ \hline
		$\tau_{7}$&$0$&$0$&$q_j^{-3}q_k$&$0$&$q_j^{-3}q_k$&$0$&$q_j^{-1}$&$q_j^{-1}$\\ \hline 	
		$\tau_8$ &$q_j^{-3}q_k$&$q_j^{-1}$&$0$&$q_j^{-3}q_k$&$0$&$q_j^{-1}$&$0$&$0$\\ \hline 
		$\tau_9$ &$0$&$0$&$0$&$q_j^{-5}q_k^2$&$q_j^{-5}q_k$&$q_j^{-1}q_k$&$0$&$q_j^{-3}$\\ \hline 
		$\tau_{10}$ &$q_j^{-3}q_k^2$&$q_j^{-1}q_k$&$q_j^{-3}q_k$&$0$&$0$&$0$&$q_j^{-1}$&$0$\\ \hline 
				
		\hline
	\end{tabular}
\end{center}

In the case of $j=0$, there are  additional possibilities, because two of the $x_i$ can be adjacent.  These possibilities are
\begin{itemize}
	\item $\tau_{11}=	(w,x \rightarrow  y,z)$
	\item $\tau_{12}=(w \rightarrow x \rightarrow y,z)$
	\item $\tau_{13}=(x \rightarrow y \rightarrow w,z)$
	\item $\tau_{14}=(x \rightarrow y,w,  z)$	
	\item $\tau_{15}=(x \rightarrow y,w \rightarrow z)$
	\item $\tau_{16}=(x \rightarrow y,z \rightarrow w)$
	\item $\tau_{17}=(x \rightarrow y,z, w)$
	\item $\tau_{18}=(w,x, y \rightarrow z)$	
	\item $\tau_{19}=(w \rightarrow x, y \rightarrow z)$	
	\item $\tau_{20}=(x \rightarrow w, y \rightarrow z)$
	\item $\tau_{21}=(x, w, y \rightarrow z)$	
	\item $\tau_{22}=(x,w \rightarrow y \rightarrow z)$	
	\item $\tau_{23}=(x, y \rightarrow z \rightarrow w)$	
	\item $\tau_{24}=(x, y \rightarrow z, w)$	,
\end{itemize}
\noindent with appropriate $Q$. 

 For the case $k=n$, we also have some special cases.

\begin{itemize}
	\item $\tau_{25}$ $x \rightarrow w \rightarrow y,z$
	\item $\tau_{26}$ $x,y \rightarrow w \rightarrow z$
	\item $\tau_{27}$ $x, y \rightarrow w, z \rightarrow w$
	\item $\tau_{28}$ $ w \rightarrow x, w \rightarrow y, z$
	\item $\tau_{29}$ $ x, w \rightarrow y, w \rightarrow z$
	\item $\tau_{30}$ $x \rightarrow w$, $y \rightarrow w, z$
\end{itemize}

To finish the proof involves checking  the coefficient of change for each $\tau_t$ and each operation. In  all these cases, some of the coefficients are zero.

In the case of adjacent addable nodes, we must alter Lemma \ref{divided} because the factors $1+q_j^2$ are needed to get the divided powers. However, whenever we have a divided power which for two non-adjacent $j$-nodes, we must use the previous Lemma \ref{divided}, and when we have an adjacency condition like $x \rightarrow y$ but use $y$ in a divided power with a different $j$-node, we not only use Lemma \ref{divided}, we must also multiply by $1+q^2$, which makes the calculations complicated. Whenever there is an entry which is a sum, it results from such a multiplication by $1+q^2$.

\begin{lem}\label{divided2}
	For $\Lm = \Lm_i$, $i \neq 0$, let $\mu$ be a partition and suppose that $j = 0$.
	\begin{enumerate}
		\item If $\mu$ has addable, adjacent $j$-nodes $x,y$ with $x$ before $y$  and   let $Q_{x,y}(\mu)=q_j^{N_j^-(x,\mu)+N_j^-(y,\mu^x)}$.
		Then in $f_j^{(2)}(\mu)$, the coefficient of $\mu^{x,y}$ is 
		\[
		q_jQ_{x,y}.
		\]
		\item If $\mu$ has three addable $j$-nodes $x,y,z$ with $x$ before $y$ and $y$ before $z$, define 
		$Q_{x,y,z}(\mu)=q_j^{N_j^-(x,\mu)+N_j^-(y,\mu^x)+N_j^-(z,\mu)}$ if $x \rightarrow y$  or 	$Q_{x,y,z}=q_j^{N_j^-(x,\mu)+N_j^-(y,\mu)+N_j^-(z,\mu^y)}$ if $y \rightarrow z$
		Then in $f_j^{(3)}(\mu)$, the coefficient of $\mu^{x,y,z}$ is 
		\[
		q_j^{-1}Q_{x,y,z}.
		\]
	\end{enumerate}
\end{lem}
\begin{proof}
	The main difference from the previous $j=0$ case is that we need the $1+q^2$ factor for the divided product. We illustrate by doing one case, that with $y \rightarrow z$.
	\begin{align*}
		f_j^{3}(\mu)=&(q_j^{N_j^-(x,\mu)+N_j^-(y,\mu^x)+N_j^-(z,\mu^{x,y})}+
			q_j^{N_j^-(y,\mu)+N_j^-(x,\mu^y)+N_j^-(z,\mu^{x,y})}\\
		+&q_j^{N_j^-(y,\mu)+N_j^-(z,\mu^y)+N_j^-(x,\mu^{y,z})})(1+q_j^2)\mu^{x,y,z}+\dots\\
		=&Q(q_j^{-4}+q_j^{-2}+1)(1+q_j^2)\mu^{x,y,z}+\dots.\\
	\end{align*}
	 To get the divided power we divide by $[3]_{q_j}[2]_{q_j}=(q_j^{-2}+1+q_j^{2})(q_j^{-1}+q_j)$ and get the desired result.
\end{proof}

 As an illustration, we show the work for $\tau_{12}=(w \rightarrow x \rightarrow y,z)$. Since we must have $w$ before we can add $x$ or $y$, we see immediately that $A=0$, $B=0$, $C_1=C_2=0$, so we need only calculate $C_3$ and $D$. 
\begin{align*} 
	f_j^{(2)}f_kf_j(\lm)=&q_j^{N(z,\lm)}q_k^{N_k(w,\lm^z)}q_jQ_{x,y}(\lm^w)  \cdot\lm_{11}+\dots\\
	=&q_jQ \cdot \lm_{11}+\dots\\
	f_j^{(3)}f_k(\lm)=&q_k^{N_k^-(w,\lm)} q_j^{-1}Q_{x,y,z}(\lm^w)  \cdot \lm_{11}\\
		=&q_jQ \cdot \lm_{11}+\dots,
\end{align*}
The last equality arises from $N_j^-(z,\lm^w)=N_j^-(z,\lm) +2$.
\end{itemize}
\end{proof}	
\begin{center}
	\begin{tabular}{ | l || c | c | c | c | c | c | c | c | }
		\hline
		$\tau$ &$A$ & $B_1$&$B_2$&$B_3$ & $C_1$&$C_2$&$C_3$&$D$ \\ \hline \hline
		$\tau_{11}$  &$q_j^{-1}$&$0$&$q_j+q_j^{3}$&$q_j^{-1}$&$q_j+q_j^{3}$&$0$&$q^5$&$q^5$\\ \hline
		$\tau_{12}$  &$0$&$0$&$0$&$0$&$0$&$0$&$q_j$&$q_j$\\ \hline
		$\tau_{13}$  &$q_j^{-1}$&$0$&$0$&$q_j^{-1}$&$0$&$0$&$0$&$0$\\\hline
		$\tau_{14}$  &$q_k^2q_j^{-1}$&$0$&$q_kq_j^{-1}+q_kq_j$&$q_k^2q_j^{-1}$&$q_kq_j^{-1}+q_kq_j$&$0$&$q_j$&$q_j$\\ \hline
		$\tau_{15}$  &$0$&$0$&$0$&$q_k^2q_j^{-3}$&$q_kq_j^{-3}+q_kq_j^{-1}$&$0$&$0$&$q_j^{-1}$\\ \hline
		$\tau_{16}$  &$q_k^2q_j^{-1}$&$0$&$q_kq_j^{-1}+q_kq_j$&$0$&$0$&$0$&$q_j$&$0$\\ \hline
		$\tau_{17}$  &$q_k^3q_j^{-1}$&$0$&$q_k^2q_j^{-1}+q_k^2q_j$&$q_k^2q_j^{-3}$&$q_kq_j^{-3}+q_k^2q_j^{-1}$&$0$&$q_kq_j$&$q_j^{-1}$\\ \hline 	
		$\tau_{18}$  &$q_j^{-1}$&$q_j^3$&$0$&$q_j^{-1}+q_j$&$q_j$&$q_j^3+q_j^5$&$0$&$q_j^5$\\ \hline 
		$\tau_{19}$   &$0$&$q_j$&$0$&$0$&$0$&$q_kq_j^{-1}+q_kq_j$&$0$&$q_k^2q_j^{-1}$\\ \hline 
		$\tau_{20}$  &$q_j^{-1}$&$0$&$0$&$q_j^{-1}+q_j$&$q_j$&$0$&$0$&$0$\\ \hline 
		$\tau_{21}$  &$q_kq_j^{-1}$&$q_j$&$0$&$q_kq_j^{-1}+q_kq_j$&$q_kq_j$&$q_kq_j^{-1}+q_kq_j$&$0$&$0$\\ \hline
		$\tau_{22}$  &$0$&$0$&$0$&$0$&$q_kq_j$&$0$&$0$&$q_j^3$\\ \hline 	
		$\tau_{23}$  &$q_kq_j^{-1}$&$q_j$&$0$&$0$&$0$&$0$&$0$&$0$\\ \hline
		$\tau_{24}$  &$q_k^3q_j^{-1}$&$q_k^2q_j$&$0$&$q_j^{-1}+q_j$&$q_kq_j^{-3}$&$q_j^{-1}+q_j$&$0$&$q_j^{-1}$\\ \hline 	 \hline 

	\end{tabular}
\end{center}

For $j=n-1, k=n$:

\begin{center}
	\begin{tabular}{ | l || c | c | c | c | c | c | c | c | }
		\hline
		$\tau$ &$A$ & $B_1$&$B_2$&$B_3$ & $C_1$&$C_2$&$C_3$&$D$ \\ \hline \hline
		$\tau_{25}$  &$0$&$0$&$q_kq_j^{-1}$&$0$&$q_kq_j^{-1}$&$0$&$0$&$0$\\ \hline 
		$\tau_{26}$  &$0$&$0$&$0$&$q_kq_j^{-3}$&$0$&$q_j^{-1}$&$0$&$0$\\ \hline
		$\tau_{27}$  &$q_kq_j^{-3}$&$q_j^{-1}$&$0$&$0$&$0$&$0$&$0$&$0$\\ \hline 	
		$\tau_{28}$  &$0$&$0$&$0$&$0$&$0$&$0$&$q_j^{-1}$&$q_j^{-1}$\\ \hline
		$\tau_{29}$  &$0$&$0$&$0$&$0$&$q_kq_j^{-5}$&$0$&$0$&$q_j^{-3}$\\ \hline
		$\tau_{30}$  &$q_j^{-3}$&$0$&$0$&$q_k^{-3}$&$0$&$0$&$0$&$0$ \\ 
		\hline 
		\hline 
	\end{tabular}
\end{center}

We summarize the results of these three propositions in the following theorem:

\begin{thm*}
	Let $\mathfrak{g}$ be the affine Lie algebra of type $A^{(2)}_{2n}$ for a positive integer $n$. The level one Fock spaces $\mathcal{F}_i$ given by $i$-corner partitions are modules for the quantum enveloping algebra
	$U_q(\mathfrak{g})$.
\end{thm*}	
\begin{proof}
	For $i=0$, this was done by Leclerc and Thibon \cite{LT}, using strict partitions and relying on the theory of perfect crystals of Kashwara et al. in \cite{KMPY}. For all $i \neq 0$, this results from Proposition \ref{first3} \ref{S.2}, \ref{S.3}. Alternatively, one can use Prof. \ref{first3} and Appendix B of \cite{KMPY}.
\end{proof}
\begin{examp}
 Let us try to calculate the results of applying the quantized Serre relations, using  the conjectural  canonical basis elements calculated in \cite{A-O}.  
\begin{itemize}
	\item \textbf{Case i=1,j=2, $a_{ij}=-1$} $\ell=1-a_{ij}=2$.  The first case where applying this relation is not trivial is content $(0,0,1,1)$, defect $1$.  We let $v$=\young(3,2)+$q^2$\young(32) be the canonical basis vector as above. Then
	\begin{align*}
	f_2(v)=&(q^{-2}+q^2)\young(32,2)\\
	f_1^{(2)}f_2(v)=&(q^{-2}+q^2)\young(321,2,1)
	\end{align*}
	\begin{align*}
	f_1(v)=& \young(3,2,1)	+q^2\young(321)\\
	f_2f_1(v)=&\young(32,2,1)	+q^2\young(321,2)\\
	f_1f_2f_1(v)=&(q^{-2}+q^{2})\young(321,2,1)
	\end{align*}
Since $f_1^{(2)}(v)=0$, this shows that the result of applying the Serre relation
\[
f_1^{2}f_2-f_1f_2f_1+f_2f_1^{2}
\]
to $v$ is $0$.
\item\textbf{Case i=0,j=1, $a_{ij}=-2$,$\ell=1-a_{ij}=3$}: We apply this for the unique canonical basis vector $w$ with content $(0,1,2,1)$ and defect 1, given by
\[
w=\young(32,2,1)+q^2\young(321,2)
\] 
Since $f_0^{(3)}(w)=0$, we have only three terms to compute.

\begin{align*}
	f_1(w)=&(q^{-2}+q^2)\young(321,2,1)\\
	f_0^{(3)}f_1(w)=& (q^{-2}+q^2)(\young(3210,2,1,0,0)+q^2 \young(32100,2,1,0))\\
	f_0^{(2)}f_1f_0(w)=&(q^{-2}+1+q^{2})(\young(3210,2,1,0,0)+q^2 \young(32100,2,1,0))\\
	f_0f_1f_0^{(2)}(w)=&f_0 (  \young(32,2,1,0,0,1)+ q^2 \young(321,2,1,0,0)	+q^2\young(32100,2,1)	+q^4\young(321001,2))\\
	=&\young(3210,2,1,0,0)+q^2\young(32100,2,1,0).
\end{align*} 
Thus altogether, $f_0^{(3)}f_1(w)-f_0^{(2)}f_1f_0(w)+f_0f_1f_0^{(2)}(w)=0$.  
\end{itemize}
\end{examp}

\noindent Keywords:  affine Lie algebra, highest weight representation, spin blocks, cyclotomic Hecke algebras

\noindent Conflict of Interest: None\\
\noindent Data availability: No data\\
\noindent Ola Amara-Omari
ORCID: 0000-0002-0651-5123

\noindent Mary Schaps
ORCID: 0000-0002-9868-6244

\end{document}